\documentclass[12pt]{amsart}
\usepackage{amsmath,amsxtra,amssymb,latexsym, amscd,amsthm,amsfonts}
\newtheorem{thm}{Theorem}[section]
\newtheorem{lem}[thm]{Lemma}
\newtheorem{prop}[thm]{Proposition}
\theoremstyle{definition}
\newtheorem{defn}[thm]{Definition}
\theoremstyle{remark}
\newtheorem{rem}[thm]{Remark}
\numberwithin{equation}{section}
\newcommand{\R}{\mathbb R}
\newcommand{\N}{\mathbb N}

\newcommand{\simuleq}[1]{\left\{\begin{aligned}#1\end{aligned}\right.}%

\begin{document}

\title[]{LAYERED VISCOSITY SOLUTIONS OF NONAUTONOMOUS HAMILTON-JACOBI EQUATIONS: SEMICONVEXITY AND RELATIONS TO CHARACTERISTICS }

\author{NGUYEN HOANG}\thanks{The research of Nguyen Hoang was partially supported by
the NAFOSTED, Vietnam.}

\author {NGUYEN MAU NAM}\thanks{The research of Nguyen Mau Nam was partially supported by
the Simons Foundation under grant \#208785.}

\address{Department of Mathematics, College of Education, Hue University, 3 LeLoi, Hue, Viet Nam}%
\email{nguyenhoanghue@gmail.com       or: hoangboi2000@yahoo.com}%
\address{Department of Mathematics, University of
Texas-Pan American, Edinburg, TX 78539, USA}
\email{nguyenmn@utpa.edu}
%


\dedicatory{}%
\vskip0.5cm
\begin{abstract}

We construct an explicit representation of viscosity solutions of
the Cauchy problem for the Hamilton-Jacobi equation $(H,\sigma)$ on
a given domain $\Omega= (0,T)\times \R^n.$ It is known that, if the
Hamiltonian $H = H(t,p)$ is not a convex (or concave) function in
$p$, or $H(\cdot, p)$ may change its  sign on $(0,T)$, then the
Hopf-type formula  does not define a viscosity solution on $\Omega.$
Under some assumptions for $H(t,p)$ on the subdomains $(t_i,
t_{i+1})\times \R^n\subset \Omega$, we are able to arrange ``partial
solutions'' given by the Hopf-type formula to get a viscosity
solution on $\Omega.$ Then we study the semiconvexity of the solution as well as its relations to characteristics.
\end{abstract}
\maketitle
\vskip0.5cm

\section{Introduction}
This paper is devoted to constructing a representation formula for
viscosity solutions of the Cauchy problem for the Hamilton-Jacobi
equation $(H,\sigma)$ of the form
\begin{equation}\label{1.1}\frac{\partial u}{\partial t} + H(t,x,D_x u)=0\, , \,\, (t,x)\in \Omega=(0,T)\times \R^n,
\end{equation}
\begin{equation}\label{1.2}
u(0,x)=\sigma(x)\, , \,\, x\in \R^n.
\end{equation}

The classical  Hopf-Lax-Oleinik formula plays an important role in
studying properties of solutions of problem (1.1)-(1.2), where
$H=H(p).$ Some generalized versions of these formulas have been
established, and they are useful tools for investigating problems of
variational calculus, differential games, etc. In most cases, the
data used to construct the formulas are concerned with the convexity
(or concavity) in the global setting. This may meet requirements of
calculus of variation or optimal control problems. Until now, no
representation formula of global solutions of Hamilton-Jacobi equations 
without concerning the convexity/concavity or its related versions
has been constructed.  Nevertheless, in the theory of PDEs, the
convexity posed on the given data is not an obligation in nature.\vspace*{0.05in}

In this paper, we consider nonautonomous Hamilton-Jacobi equations
in a new situation. We suppose that the behavior of the Hamiltonian
$H(t,x,p)$ may vary on each subdomain $[t_i, t_{i+1}]\times
\R^n\times\R^n$ of the domain $[0,T]\times\R^n\times\R^n$, where $$0=
t_1<t_2<\dots <t_k=T.$$ For example, $H(t,x,\cdot)$ may be convex for all
$(t,x)\in[t_{i-1},t_i]\times \R^n$, and then $H(t,x,\cdot)$ changes
into a concave function in the successive subdomain $(t,x)\in [t_i,
t_{i+1}]\times \R^n.$ \vspace*{0.05in}

The paper is structured as follows. In section 2, we introduce a way
to joint viscosity solutions of problem (1.1)-(1.2) on subregions of
the form $[t_i,t_{i+1}]\times\R^n$ in order to obtain a global
solution called a ``layered viscosity solution'' of the problem as a
whole. In section 3, we examine the case where the Hamiltonian
$H=H(t,p)$ is continuous and the initial data $\sigma(x)$ is convex.
Under some assumptions on the given data, we use the Hopf-type
formula for ``partial solutions''
\begin{equation*}\label{1.3}u_i(t,x) =\max_{q\in \R^n}\, \{\langle x,q\rangle -\, \sigma_i^*
(q)-\int_{t_i}^t H(\tau,q)d\tau\},\end{equation*} to build a
``layered viscosity solution'' of the Cauchy problem on the whole
domain. We also prove that the obtained viscosity solution inherits
the semiconvexity  analogous to ``partial solutions''. In section 4,
we study relations between the Hopf-type formula and the ``layered
viscosity solution'' as well as its singularity  based on the
semiconvexity of the solution $u(t,x)$ in connection with the
characteristics.\vspace*{0.05in}

Our notations are standard in the field of nonlinear PDEs. For a positive real number $T$, denote
$\Omega =(0,T)\times \R^n.$ Let $\, |\, .\, |$ and $\langle
.,.\rangle$ be the Euclidean norm and the scalar product in $\R^n$,
respectively. For a function $u:\ \Omega \to \R,$ we denote by
$D_xu$ the gradient of $u$ with respect to variable $x$, i.e., $D_xu
=(u_{x_1},\dots, u_{x_n}).$ \vspace*{0.05in}

A continuous functions $u: [0, T]\times \R^n\to \R$ is called a {\it viscosity solution} of the Cauchy problem (1.1)-(1.2) on $\Omega$ provided that the following hold:\\[1ex]
{\rm (i)} $ u(0,x)=\sigma(x)$ for all $x\in \R^n$;\\
{\rm (ii)} For each $v\in C^1( \Omega)$, if $u-v$ has a local maximum  at a point $(t_0,x_0)\in \Omega,$ then
$$v_t(t_0,x_0) +H(t_0,x_0,D_xv(t_0,x_0)) \le 0,$$
and if $u-v$ has a local minimum  at a point $(t_0,x_0)\in \Omega,$ then
$$ v_t(t_0,x_0) +H(t_0,x_0,D_xv(t_0,x_0)) \ge 0.$$

\section{``Layered viscosity solution" in general cases}

Consider the partial differential equation $F(x,u,Du)=0$ on a domain
$\mathcal O \subset \R^m.$ Suppose that $\mathcal O$ can be divided in two open subsets $\mathcal O_1$ and $\mathcal O_2$ by a $C^1$
surface  $\Gamma$, that is $\mathcal O=\mathcal O_1\cup\mathcal
O_2\cup \Gamma$, then several compatible conditions can be added on
$\mathcal O_i$ and $\Gamma$, as well as $C^1$-solutions  $u_i$  of
the equation $F(x,u, Du)=0$ on $\mathcal O_i\cup \Gamma$ for
$i=1,2$, to get a viscosity solution on whole $\mathcal O$; see
\cite[Theorem~1.3]{[3]}. However, in the case of the Cauchy problem,
the situation seems to be rather simple as follows.

\begin{thm}
Let $H(t,x,p)$ be a continuous function on $[0,T]\times \R^n\times \R^n$ and let $\sigma_1(x),\ \sigma_2(x)$ be continuous functions on $\R^n.$ For $t_1\in (0,T)$, consider the following Cauchy problems:
\begin{equation}\label{I}\simuleq{&u_t +H(t,x,D_xu)=0,\quad (t,x)\in (0,t_1)\times \R^n, \\
&u(0,x)=\sigma_1(x),\quad x\in \R^n,}\end{equation} and
\begin{equation}\label{II}\simuleq{&u_t +H(t,x,D_xu)=0,\quad (t,x)\in (t_1,T)\times \R^n, \\
&u(t_1,x)=\sigma_2(x),\quad x\in \R^n.}\end{equation}

Suppose that problems {\rm (2.1)} and {\rm (2.2)} have viscosity solutions $u_1(t,x)$ and $u_2(t,x)$ on the domains $\Omega_1=(0,t_1)\times\R^n$ and $\Omega_2=(t_1,T)\times \R^n$, respectively, where $\sigma_2(x)=u_1(t_1,x)$ for all $x\in \R^n$. Then the Cauchy problem
\begin{equation}\label{II}\simuleq{&u_t +H(t,x,D_xu)=0,\quad (t,x)\in \Omega=(0,T)\times \R^n, \\
&u(0,x)=\sigma_1(x),\quad x\in \R^n}\end{equation}
has a viscosity solution of the form
\begin{equation}\label{II1} \tilde u(t,x)=\simuleq{&u_1(t,x), \quad  (t,x)\in [0,t_1]\times \R^n, \\
&u_2(t,x), \quad (t,x)\in [t_1,T]\times \R^n.}\end{equation}
\end{thm}

\begin{proof}
Under the assumptions made, the function $\tilde u$ defined by (\ref{II1})
is continuous. We will use the {\it ``extrema at a terminal time"}
property of the Cauchy problem. Let $(t_0,x_0)\in \Omega$ and $v\in
C^1(\Omega).$ Assume that $\tilde u -v$ has a local maximum at
$(t_0,x_0).$ If $(t_0,x_0)\in \Omega_i, \ i=1,2,$ then the
restriction of $\tilde u(t,x)$ on $\Omega_i$ coincides with
$u_i(t,x)$. Thus, the inequality
$$v^i_t(t_0,x_0) +H(t_0,x_0,D_xv^i(t_0,x_0)) \le 0,$$
holds since $u_i(t,x)$ is a viscosity subsolution on $\Omega_i$, and
$v^i=v\big |_{\Omega_i}$ is also a test function such that $u_i-v^i$
attains a local maximum at $(t_0,x_0).$ It follows that
$$v_t(t_0,x_0) +H(t_0,x_0,D_xv(t_0,x_0)) \le 0.$$
Now, let $t_0=t_1$ and let $\tilde u-v$ have a local maximum at
$(t_1,x_1)\in \Omega.$ It is obvious that $u_1-v$ also attains a
local maximum at $(t_1,x_1)$ on the set $(0,t_1]\times \R^n.$
Arguing as in the proof of the \cite[Lemma~4.1]{[3]}, we see that
$$v_t(t_0,x_0) +H(t_0,x_0,D_xv(t_0,x_0)) \le 0,$$
that is, $\tilde u(t,x)$ is a viscosity subsolution of the problem. Similarly, we are able to verify that $\tilde u(t,x)$ is a viscosity supersolution.
\end{proof}

We call a solution of problem (1.1)-(1.2) given by (2.4) a {\it
layered viscosity solution} of the problem. In the lemma below, we
will show that a translation of coordinates does not change the
properties of viscosity solutions.

\begin{lem}
Suppose that $H(t,x,p)$ is a continuous function on $[t_1,T]\times
\R^n\times \R^n,\ u(t,x)$ is a continuous function on $[t_1,T]\times
\R^n$, and $\sigma(x)$ is a continuous function on $\R^n.$ Let
\begin{equation*}
H_1(\tau,x,p)=H(t_1+\tau,x,p),\ u_1(\tau,x)=u(t_1+\tau,x), \tau
\in[0,T-t_1]
\end{equation*}
for $x, p\in\R^n$. Then $u(t,x)$ is a viscosity solution of the
Cauchy problem $(H,\sigma)$ on $(t_1,T)\times\R^n$ if and only if
$u_1(\tau,x)$ is a viscosity solution of problem $(H_1, \sigma)$ on
$(0,T-t_1)\times \R^n$.
\end{lem}

\begin{proof}
 Observe that if $v(t,x)\in C^1([t_1,T]\times\R^n)$, then $v_1(\tau,x)=v(t_1+\tau,x)$ is of class $C^1([0,T-t_1]\times\R^n)$. Thus, the proof of the lemma is straightforward using the definition of a viscosity solution.
\end{proof}

\begin{rem}
Suppose that the Hamiltonian $H=H(t,x,p):\ \mathcal O \to \R$ is
continuous, where $\mathcal O=[0,T]\times\R^n\times\R^n.$ Let
$t_1\in (0,T)$ and put $\mathcal O_1=[0,t_1]\times\R^n\times\R^n,$
$\mathcal O_2=[t_1,T]\times\R^n\times\R^n.$ Let $H_1$ (resp. $H_2$)
be the restriction of $H$ on $\mathcal O_1$ (resp. $\mathcal O_2$).
If $H(t,x,p)$ is a convex or concave function with respect to the
third variable for all $(t,x)\in [0,T]\times \R^n,$ a viscosity
solution of the problem $(H,\sigma)$ can be obtained by methods of
variational calculus under some assumptions; see [2] and the
references therein. If the convexity of $H(t,x,p)$ in $p$ is
changed, for example, $H(t,x,p)$ is convex on $\mathcal O_1$ and
concave on $\mathcal O_2$, then the afore-mentioned method cannot
be applied on the whole domain $[0,T]\times \R^n$. Nevertheless, we
can get a \emph{layered viscosity solution} for problem $(H,\sigma)$
using Theorem 2.1
\end{rem}

\section{``Layered viscosity solution" based on Hopf-type formulas}

In this section, we pay attention to the case where the initial data $\sigma (x)$ is convex and $H=H(t,p)$ is a continuous function satisfying some additional conditions.
\subsection{Assumptions}

We now  consider the  Cauchy problem for the Hamilton-Jacobi equation:
 \begin{equation}\label{2.1}\frac{\partial u}{\partial t} + H(t, D_x u)=0\, , \,\,
(t,x)\in \Omega =(0,T)\times \R^n,\end{equation}
\begin{equation}\label{2.2}u(0,x)=\sigma(x)\, ,
\,\, x\in \R^n,\end{equation} where the Hamiltonian $H(t,p)$ is of
class $C([0,T]\times\R^n)$,  and $\sigma (x)\in C(\R^n)$ is a convex
function.\vspace*{0.05in}

Let $\sigma^*$ be the Fenchel conjugate of $\sigma.$ By definition,
\begin{equation*}
\sigma^*(x)=\displaystyle\sup_{q\in \R^n} \{ \langle x,q\rangle -\sigma(q)\}.
 \end{equation*}
 We denote by $$D=  \mbox{dom}\, \sigma^*=\{y\in \R^n\ |\, \sigma^*(y)<+\infty\}$$ the effective domain of the convex function $\sigma^*$. The \emph{Hopf-type formula} for the Cauchy problem (3.1)-(3.2) is given by
\begin{equation}\label{2.3}u(t,x) =\max_{q\in \R^n}\, \{\langle x,q\rangle -\, \sigma^*
(q)-\int_0^t H(\tau,q)d\tau\}=(\sigma^*+\int_0^tH(\tau,\cdot)d\tau)^*(x).\end{equation}

To prove that the Hopf-type formula (3.3) is a viscosity solution of
problem (3.1)-(3.2), we use the ``consecutiveness property" (or
``$\mathcal C$- property" for short) of viscosity solution (see [6]): for all $t,s$ such that $0\le s\le t\le T,$
\begin{equation*} \label{1.4}\Big[ \big(\sigma^* +\int_0^sH(\tau, \cdot)d\tau\big)^{**} +\int_s^tH(\tau,\cdot)d\tau\Big]^*=\Big[\sigma^* +\int_0^t H(\tau,\cdot)d\tau\Big]^*.\end{equation*}

This relation is true if condition (A2) (to be formulated later) for $H(t,p)$ holds. It concerns mainly with the assumption that the Hamiltonian $H(\cdot, p)$ does not change its sign on $(0,T).$ If this condition is violated, then the Hopf-type formula (3.3) is no longer a viscosity solution  of problem (3.1)-(3.2); see also \cite{[5],[6]}.
Therefore, we will use Theorem 2.1 to overcome the difficulty. Suppose that $H(\cdot, p)$ may change its sign several times on $(0,T)$. Then on each subinterval $(t_i,t_{i+1})$, where $H(\cdot, p)$ does not change its sign, we use Hopf-type formula to get a viscosity solution $u_i(t,x).$  Then we arrange consecutively these ``partial viscosity solutions" to obtain a layered viscosity solution of the problem on whole interval $(0,T)$. This solution also satisfies the semiconvexity as in the case of the Hope-type formula.\vspace*{0.05in}

Let us recall some results necessary in further. First, we need the following assumptions on the domain $[a, b]\times \R^n$, $a<b$:\\[1ex]
(A1) {For every $(t_0,x_0)\in
[a,b]\times\R^n $,  there exist positive constants $r$ and $N$
such that
$$\langle x,p\rangle -\, \sigma ^* (p)-\int_a^t H(\tau,p)d\tau <
 \max_{|q|\le N}\{\langle x,q\rangle -\, \sigma ^* (q)-\int_a^t H(\tau,q)d\tau\},$$
whenever $ (t,x)\in [a,b]\times\R^n,\, |t-t_0|+|x-x_0|<r$ and
$| p| >N.$}\\[1ex]
(A2) $H(t,p)$ is admitted as one of two following forms:\\[1ex]
a) $H(t,\cdot) $ is a convex (or concave) function for all $t\in [a,b].$\\
b) $H(t,p)=g(t)h(p) +k(t)$ for some functions $g,\ h,\  k$, where $g(t)$ does not change its sign for all $t\in [a,b].$\vspace*{0.05in}

For each $(t,x)\in \Omega,$ let $\ell (t,x)$ be the set of all $p\in
\R^n$ at which the function $$q\mapsto \langle x,q\rangle -\,
\sigma^* (q)-\int_a^t H(\tau,q)d\tau$$ attains a maximum. In virtue
of (A1) on $[a,b]\times \R^n$, one has $\ell(t,x)\ne \emptyset.$\vspace*{0.05in}

Then we have the following theorem.

 \begin{thm} {\rm([6])}. Assume that {\rm (A1)} and {\rm (A2)} are satisfied on $[0,T]\times \R^n$. Moreover, let $H(t,p) \in C^1([0,T]\times \R^n).$ Then the Hopf-type formula
$u(t,x)$ defined by {\rm(\ref{2.3})}
 is a viscosity solution of the Cauchy problem {\rm(3.1)-(3.2)}.
\end{thm}

As a  direct consequence of  Lemma 2.2, the following proposition is used to define the Hopf-type formula for the initial time $t_0\ne 0.$

\begin{prop} Let $ 0<t_1<T.$
Suppose that $\sigma(x)$ is a convex function on $\R^n$, and
$H(t,p)$ is a continuously differentiable function on $[t_1,T]\times \R^n$ satisfying
conditions {\rm(A1)} and {\rm(A2)} on $[t_1, T]\times \R^n$. Then the function $$u(t,x) =\big(\sigma^* +\int_{t_1}^{t}
H(\tau,\cdot)d\tau\big)^* (x)$$ is a viscosity solution of problem
$(H,\sigma)$ on $(t_1,T)\times \R^n.$
\end{prop}

The following theorem defines a viscosity solution of the Cauchy problem for Hamilton-Jacobi equation $(H,\sigma)$ where condition (A2) holds on each subinterval of the $(0,T).$
\begin{thm}
Let $\sigma(x)$  be a convex function on $\R^n$ and $H(t,p)\in C^1([0,T]\times \R^n).$ Suppose that there exists $t_1\in (0,T)$ such that on each subdomain $[0,t_1]\times \R^n$ and $[t_1,T]\times \R^n$, conditions {\rm(A1)} and {\rm(A2)} are satisfied, where $\sigma(\cdot)$ is replaced by $u(t_1,\cdot)$ for condition {\rm (A1)} on $[t_1, T]\times \R^n$. Then the function $u(t,x)$ given by

\begin{equation*}\label{II} u(t,x)=\simuleq{&\big (\sigma^* +\int_0^t H(\tau,\cdot)d\tau\big)^* (x), \quad  (t,x)\in [0,t_1]\times \R^n, \\
&\big (u^*(t_1,\cdot) +\int_{t_1}^t H(\tau,\cdot)d\tau\big)^*(x),\quad (t,x)\in [t_1,T]\times \R^n}\end{equation*}
is a viscosity solution of the Cauchy problem $(H,\sigma)$ on $(0,T)\times\R^n.$
\end{thm}

\begin{proof} We note first that, for fixed $t_1\in (0,T)$, the function $\omega(\cdot) =u(t_1,\cdot)$ is convex. By Theorem 3.1 and Proposition 3.2, and
 then applying Theorem 2.1, we deduce the conclusion of the theorem.
\end{proof}

\begin{rem}
From Theorem 3.3, we have several corollaries. For example, we can suppose that the interval $[0,T]$ can be split into  $m$ subintervals $[t_i, t_{i+1}]$ on which\\[1ex]
{\rm (i)} $H(t,\cdot)$ is convex (or concave), or\\
{\rm (ii)} If $H(t,x)=g(t)h(p)+k(t)$ then $g(t)$ does not change its sign on each subinterval.\vspace*{0.05in}

Of course, on $[0,T]$, the Hamiltonian $H(t,p)$ can be mixed by forms (i) or (ii) from one subinterval to another. Then a viscosity solution $u(t,x)$ on $[0,T]\times\R^n$ can be defined as $m-$``layered viscosity solution" via the solutions $u_i(t,x)$ of the problem $(H,\sigma_i)$ on $(t_i,t_{i+1})\times \R^n$ given by Hopf-type formula, where $\sigma_i(x) =u_i(t_i,x),\  \sigma_0(x)=\sigma(x)$ are convex functions.
\end{rem}
\subsection{The semiconvexity of layered viscosity solutions}

\begin{defn}
A function $v:\ \Omega \to \R$ is called {\it semiconvex} with linear modulus if there is a constant $C>0$ such that
$$v(\lambda y_1+(1-\lambda)y_2)-\lambda v(y_1)-(1-\lambda ) v(y_2)\le \lambda (1-\lambda) \frac C2 |y_1-y_2|^2$$
for any $y_1=(t_1,x_1)$ and $y_2=(t_2,x_2)$ in $\Omega$ and for any
$\lambda \in [0,1].$
\end{defn}
The function $v$ is semiconcave if and only if $-v$ is semiconvex.
In this paper, we therefore use some properties of semiconvex
functions transferred from those of semiconcave functions given in
[2]. The theory of semiconcave functions has been developed since
the last decades of previous century. The reader is referred to the
aforementioned monograph [2] for a systematic development of the
theory.\vspace*{0.05in}

In our recent paper [6], we prove that the viscosity solution of the
Cauchy problem $(H,\sigma)$ given by Hopf type formula (3.3) is  a
semiconvex function on the domain $\Omega$. For the ``layered
viscosity solution", we also get the same property as in the theorem
below. In this theorem condition (A1) is automatically satisfied
since $\sigma$ is convex and Lipschitz on $\R^n$, which implies the
domain of $\sigma^*$ is bounded; see [9].

 \begin{thm} Suppose that $H(t,p) \in C^1([0,T]\times \R^n)$ and $\sigma$ is convex and Lipschitz on $\R^n$. Let $t_+\in (0,T)$ be a positive real number such that on each subdomain $[0,t_+]\times \R^n$ and $[t_+,T]\times \R^n$, the Hamiltonian $H(t,p)$ satisfies condition {\rm (A2)}. Then the layered viscosity solution $u(t,x)$ defined by 
\begin{equation}\label{new} u(t,x)=\simuleq{&\big (\sigma^* +\int_0^t H(\tau,\cdot)d\tau\big)^* (x), \quad  (t,x)\in [0,t_+]\times \R^n, \\
&\big(u^*(t_+,\cdot) +\int_{t_+}^t
H(\tau,\cdot)d\tau\big)^*(x),\quad (t,x)\in [t_+,T]\times
\R^n}\end{equation} 
is a semiconvex function on $\Omega.$\end{thm}

\begin{proof}  For $(t_1,x_1)$ and $(t_2,x_2)$ in $\Omega$, where $t_1
\leq t_2$, and $\lambda \in [0,1]$, put $(t_0,x_0)=\lambda
(t_1,x_1)+(1-\lambda)(t_2,x_2).$\vspace*{0.05in}

Consider the following cases:\vspace*{0.05in}

\noindent {\bf Case 1.} $0\le t_1\le t_2\le t_+$ or $t_+\le t_1\le t_2\le T.$
\smallskip

Let $D=\mbox{dom }\sigma^*$. Arguing as in the proof of Theorem 2.4 in [6], we obtain the inequality below
\begin{equation*}
u(t_0,x_0)\le \lambda u(t_1,x_1)+(1-\lambda)u(t_2,x_2)
+\lambda(1-\lambda) M|t_2-t_1|^2,\end{equation*}
where $M=\sup_{\tau\in [0,T], \ p\in D}|H_t(\tau,p)|$ by applying the
mean value theorem. Note that $D$ is bounded since $\sigma(x)$ is
Lipschitz on $\R^n$.\vspace*{0.05in}

\noindent {\bf Case 2.} $0\le t_1\le t_0\le t_+\leq t_2\le T.$
\smallskip

Take $p\in \ell(t_0,x_0)$ and let
$$\Phi \ =\ u(t_0,x_0)-\lambda u(t_1,x_1)-(1-\lambda )u(t_2,x_2).$$
Observe that
$$(u^*(t_+,\cdot) -\sigma^*(\cdot)) (p) =\big[(\sigma^*+\int_0^{t_+} H(\tau,\cdot)d\tau)^{**} (\cdot) -\sigma^*(\cdot)\big] (p)\le \int_0^{t_+} H(\tau,p)d\tau$$
since $v^{**}\le v$ for any function $v$.

It follows that
$$\aligned \Phi \ \le  &\  \langle \lambda x_1+(1-\lambda )x_2,p\rangle -\sigma^*(p)-\int_0^{t_0}H(\tau,p)d\tau\\
-& \  \lambda\langle x_1,p\rangle +\lambda \sigma^*(p)+\lambda \int_0^{t_1}H(\tau,p)d\tau\\
 - &\ (1-\lambda)\langle x_2,p\rangle +(1-\lambda) u^*(t_+, p)+(1-\lambda)\int_{t_+}^{t_2}H(\tau,p)d\tau\\
\le &\  \lambda \Big( \int_{t_0}^{t_1}H(\tau,p)d\tau \Big) +(1-\lambda) \Big( \int_{t_+}^{t_2}H(\tau,p)d\tau -\int_0^{t_0} H(\tau,p)d\tau\Big) \\
+ &\ (1-\lambda) (u^*(t_+,\cdot) -\sigma^*(\cdot)) (p)\\
\le & \  \lambda \Big( \int_{t_0}^{t_1}H(\tau,p)d\tau \Big) +(1-\lambda) \Big( \int_{t_0}^{t_2}H(\tau,p)d\tau \Big).\endaligned$$
Therefore, we have for some $\tau^*_1\in [t_1,t_0],\  \tau^*_2 \in [t_0,t_2],$
$$\aligned \Phi \ \le &\ \lambda (t_1-t_0)H(\tau^*_1,p) +(1-\lambda) (t_2-t_0)H(\tau^*_2,p)\\
\le &\ \lambda (t_1-(\lambda t_1+(1-\lambda)t_2)) H(\tau^*_1,p) +(1-\lambda) (t_2-(\lambda t_1+(1-\lambda)t_2)) H(\tau^*_2,p)\\
\le&\ \lambda (1-\lambda)(t_1-t_2) \big (H(\tau^*_1,p)-H(\tau^*_2,p)\big)\\
\le&\  \lambda (1-\lambda)|t_1-t_2| \big |H(\tau^*_1,p)-H(\tau^*_2,p)\big|\\
\le&\ \lambda (1-\lambda)\, M\, |t_1-t_2|^2,\endaligned$$
where $M$ is defined in Case 1.

Consequently,
$$u(t_0,x_0)\le \lambda u(t_1,x_1)+(1-\lambda )u(t_2,x_2) +\lambda (1-\lambda) M|t_1-t_2|^2.$$

\noindent {\bf Case 3.} $0\le t_1\leq t_+\le t_0\le t_2\le T.$
\medskip

We can assume that $t_+<t_0$ because the other possibility has been considered in Case 2. Since $t_0=\lambda t_1+(1-\lambda)t_2,$ one has $\lambda =\frac{t_2-t_0}{t_2-t_1}.$ The line segment joining $(t_1,x_1)$ and $(t_2,x_2)$ intersects the plane $t=t_+$ at $(t_+,x_+).$ Let $t_0 =\mu t_++(1-\mu) t_2$ and $t_+ =\alpha t_1 +(1-\alpha)t_0$. Then
\begin{equation}\label {nN}\mu =\frac{t_2-t_0}{t_2-t_+} \  \ \text{and}\  \ \alpha =\frac{t_0-t_+}{t_0-t_1}.\end{equation}
A simple calculation shows that
$$\lambda =\frac{\alpha \mu}{1-\mu +\alpha \mu}\quad  \mbox{ and }\quad  1-\lambda =\frac{1- \mu}{1-\mu +\alpha \mu}.$$

Applying Cases 1 and 2 in the sequel, we have
$$u(t_0,x_0)\le \mu u(t_+,x_+)+(1-\mu )u(t_2,x_2) +\mu(1-\mu)M |t_2-t_+|^2$$
and
$$u(t_+,x_+)\le \alpha u(t_1,x_1)+(1-\alpha )u(t_0,x_0) +\alpha(1-\alpha)M |t_0-t_1|^2.$$
Therefore,
$$\aligned u(t_0,x_0)\le \alpha\mu u(t_1,x_1) +&\mu(1-\alpha)u(t_0,x_0) +\mu\alpha(1-\alpha) M|t_0-t_1|^2\\
+& (1-\mu )u(t_2,x_2) +\mu(1-\mu)M |t_2-t_+|^2.\endaligned$$
Thus,
$$\aligned u(t_0,x_0)\ \le &\ \frac{\alpha\mu}{1-\mu+\alpha \mu}u(t_1,x_1) +\frac{1-\mu}{1-\mu+\alpha \mu}u(t_2,x_2) \\
+&\  \frac{\alpha\mu (1-\alpha)}{1-\mu+\alpha \mu}M |t_0-t_1|^2 +\frac{\mu(1-\mu)}{1-\mu+\alpha \mu}M |t_2-t_+|^2\\
\le &\ \lambda u(t_1,x_1) +(1-\lambda)u(t_2,x_2) +M\big ( \lambda(1-\alpha) |t_0-t_1|^2 +\mu(1-\lambda)|t_2-t_+|^2\big).\endaligned$$

Let $\Delta =  \lambda(1-\alpha) |t_0-t_1|^2 +\mu(1-\lambda)|t_2-t_+|^2$. Using (\ref{nN}) and the fact that $\lambda=\frac{t_2-t_0}{t_2-t_1}$, we have
$$\aligned \Delta = &\ \lambda(1-\alpha) (1-\lambda)^2|t_2-t_1|^2 +\frac{1-\lambda}{\mu}|t_2-t_0|^2\\
=&\ \lambda(1-\lambda)^2(1-\alpha)|t_2-t_1|^2 +(1-\lambda)\lambda \frac{\lambda}{\mu}|t_2-t_1|^2\\
=&\  \lambda(1-\lambda) |t_2-t_1|^2\big( (1-\alpha)(1-\lambda) +\frac{\lambda}{\mu}\big )\le 2\lambda(1-\lambda) |t_2-t_1|^2\endaligned$$
since $\displaystyle\ {\frac{\lambda}{\mu} =\frac{t_2-t_+}{t_2-t_1} \in [0,1] }.$
\smallskip

Finally, we get
$$u(t_0,x_0)\le \lambda u(t_1,x_1)+(1-\lambda )u(t_2,x_2) +\lambda (1-\lambda) 2M|t_1-t_2|^2.$$
It follows from the definition that the function $u(t,x)$ is semiconvex in $\Omega.$ The theorem has been proved.
\end{proof}
\begin{rem}
1. Under the assumptions of Theorem 3.6, the viscosity solution $u(t,x)$ given by (\ref{new}) is a semiconvex function. Thus, it is also a locally Lipschitz function on $\Omega.$ For $t\in [t_+, T]$, we can rewrite $u(t,x)$ as follows:
$$u(t,x) =\Big( (\sigma^*+\int_0^{t_+}H(\tau,\cdot)d\tau)^{**} +\int_{t_+}^tH(\tau,\cdot)d\tau\Big)^* (x).$$

On the other hand, the Hopf-type formula for the Cauchy problem (3.1)-(3.2) is of the following form:
$$u_H(t,x)=\big(\sigma^* +\int_0^tH(\tau,\cdot)d\tau\big)^* (x)$$
The function $u_H(t,x)$ is defined on $[0,T]\times \R^n$, and it is also a locally Lipschitz semiconvex function; see [6, 10]. Note that
$$\big(\sigma^*+\int_0^{t_+}H(\tau,\cdot)d\tau\big)^{**}\le \sigma^* +\int_0^{t_+} H(\tau,\cdot)d\tau.$$
Therefore, $u_H(t,x)\le u(t,x)$ for all $(t,x)\in [t_+,T]\times\R^n.$

If the $``\mathcal C$-property" for $H(t,p)$ and $\sigma$ holds on $[0,T]$, then $u(t,x) =u_H(t,x)$ on $\Omega.$ Otherwise, for some $(t,x)\in \Omega$, where $t>t_+$, one has $u(t,x)>u_H(t,x).$\vspace*{0.05in}

2. It is worth noticing that both functions $u(t,x)$ and $u_H(t,x)$ are Lipschitz solutions of the Cauchy problem (3.1)-(3.2), (i.e., the solutions are locally Lipschitz and satisfy the equation at almost all points of the domain) and moreover, they are semiconvex functions. Therefore, the criteria ``semiconvexity" or ``semiconcavity" are not strong enough to get the uniqueness of Lipschitz solution of the Cauchy problem for Hamilton - Jacobi equation if the Hamiltonian $H(t,x,p)$ is not concave or convex in the variable $p$; see [4, p.132].
\end{rem}
\section{Relationship between viscosity solutions and characteristics. Singularity}

\subsection{The Hopf-type formula and characteristics}

Next, we are concerned with the function $\omega(x) =u(t_1, x)$ as the initial data for problem $(H,\omega)$ on the interval $[t_1, T].$ We will investigate the relationship of characteristic curves between two problems $(H,\sigma)$ and $(H,\omega).$ To this aim, let us recall the Cauchy method of characteristics for problem (3.1)-(3.2). Note that, to use the method of characteristics, at least the initial data is assumed to be of class $C^1$. Therefore, some  sufficient conditions for $u(t,x)$ to be of class $C^1$ in $[0,t_1]\times\R^n$ will also be presented here. \vspace*{0.05in}

From now on, we will use  the standing assumptions that $H(t,p)\in C^1([0,T]\times\R^n)$ and $\sigma\in C^1(\R^n).$  Moreover, $\sigma$ is convex and Lipschitz on $\R^n.$ \vspace*{0.05in}

The characteristic differential equations of problem
(3.1)-(3.2) are given as follows:
\begin{equation}\label{2.5}
\dot x=H_p \ ;\qquad \dot v = \ \langle H_p,p\rangle
 - \ H \ ;\qquad \dot p=0 \,\end{equation}
with the initial conditions:
\begin{equation}\label{2.6} x(0)=y \ ;\qquad v(0)=\sigma(y)\ ;\qquad p(0)=
\sigma _y(y)\ ,\quad y\in \R^n.\end{equation}

Then a characteristic strip of problem
(\ref{2.1})-(\ref{2.2}) (i.e., a solution of the system of
differential equations (\ref{2.5})-(\ref{2.6})) is defined by
\begin{equation}\label{2.7}\simuleq{x&=x(t,y)=y+\int_0^t H_p(\tau,\sigma_y(y))d\tau, \\
v&=v(t,y)=\sigma(y)+\int_0^t
\langle H_p(\tau,\sigma_y(y)),\sigma_y(y)\rangle d\tau-\int_0^tH(\tau,\sigma_y(y))d\tau,\\
 p&= p(t,y)\ =\ \sigma_y(y).}\end{equation}

The first component of solutions (\ref{2.7}) is called the characteristic curve (briefly, characteristics) emanating from $y$ i.e., the curve defined by
\begin{equation*}\label{2.8}\mathcal C:\ x=x(t,y)=y+\int_0^t H_p(\tau,\sigma_y(y))d\tau,\ t\in [0,T]. \end{equation*}
\medskip

Let $(t_0,x_0) \in \Omega. $ Denoted by $\ell^*(t_0,x_0)$ the set of all $y\in \R^n$ such that there is a
characteristic curve emanating from $y$ and passing the
point $(t_0,x_0).$ We have $\ell (t_0,x_0) \subset {\sigma}_y(\ell^*(t_0,x_0))$; see [8]. Therefore, $\ell^*(t_0,x_0) \ne \emptyset.$ Moreover,

\begin{prop}{\rm ([7])} Let $(t_0,x_0)\in \Omega.$ Then a characteristic curve $(\mathcal C)$ passing $(t_0,x_0)$ has form
\begin{equation*}\label {2.9} x=x(t,y)=x_0+\int_{t_0}^t H_p(\tau,\sigma_y(y))d\tau\end{equation*}
for some $y\in \ell^*(t_0,x_0).$
\end{prop}

We say that the characteristic curve $\mathcal C$ is of {\it type (I) } at the point $(t_0,x_0) \in \Omega$ if $\sigma_y(y) =p\in \ell(t_0,x_0).$ If $\sigma_y(y)\in \sigma_y(\ell^*(t_0,x_0))\setminus \ell(t_0,x_0)$, then $\mathcal C$ is said to be of {\it type (II) } at $(t_0,x_0).$ \vspace*{0.05in}

For $t_1\in (0,T),$ we reproduce in the theorem below a result in [7] stating that the viscosity solution $u(t,x)$ of problem $(H,\sigma)$  is of class $C^1([0,t_1]\times\R^n)$.
\begin{thm}
Let $H(t,p)$  satisfy condition {\rm(A2)} on $[0,t_1]\times\R^n$  and let $\sigma\in C^1(\R^n)$ be convex and Lipschitz.  Suppose that the function $$\omega(x)=u(t_1,x)=(\sigma^* +\int_{0}^{t_1} H(\tau,\cdot)d\tau)^*(x)$$ is of $C^1(\R^n).$ Then the solution  $$u(t,x)=(\sigma^*  +\int_{0}^{t} H(\tau,\cdot)d\tau)^*(x)$$ is of  class $C^1([0,t_1]\times \R^n).$
\end{thm}

Let $(t_0,x_0)\in (t_1,T)\times \R^n.$  Suppose that the characteristic curve $\mathcal C$
$$x=x_0+\int_{t_0}^t H(\tau,\sigma_y(y))d\tau,$$
where $y\in \ell^*(t_0,x_0),$ intersects the plane $t=t_1$ at $(t_1,x_1).$  We have the following theorem.

\begin{thm}
Let $H(t,p)\in C^1([0,T]\times\R^n)$ and let $\sigma\in C^1(\R^n)$ be a convex Lipschitz function. Suppose that for some $t_1\in (0,T)$, the function $$\omega(x)=u(t_1,x)=(\sigma^*+\int_0^{t_1} H(\tau,\cdot)d\tau)^*(x)$$ is of class $C^1(\R^n).$ Then every characteristic curve $$\mathcal C:\ x=x(t,z)=z+\int_{t_1}^tH_p(\tau,\omega_z(z))d\tau,\ t\in [t_1,T]$$ of problem $(H,\omega)$ can be extended backward to form a characteristic curve of problem $(H,\sigma)$ starting at $(0,y)$, where $y= z-\int_0^{t_1} H_p(\tau,\omega_z(z))d\tau.$

Conversely, let $\mathcal C':\ x=x(t,y)=y+\int_0^tH(\tau,\sigma_y(y))d\tau$ be a characteristic curve of type (I) of the problem $(H,\sigma)$ on $[0,T]\times\R^n.$ Then the restriction of $\mathcal C'$ on $[t_1,T]$ is also a characteristic curve of problem $(H,\omega)$ on $[t_1,T]\times\R^n.$
\end{thm}

\begin{proof}
Let $ \mathcal C_1: x=x_1 +\int_{t_1}^tH_p(\tau,\omega_z(z))d\tau, t_1\le t\le T$, be a characteristic curve of the problem $(H,w)$ starting from $(t_1,x_1).$ Since $\omega(x)=u(t_1,x)$ is differentiable at $(t_1,x_1)$ under the assumptions made, we see that $\ell(t_1,x_1)$ is a singleton and $q=\omega_z(z) \in \ell(t_1,x_1).$ Let $\mathcal C_0: x=x_1 +\int_{t_1}^t H_p(\tau,q)d\tau, 0\le t\le t_1$, be the characteristic curve of type (I) of the problem $(H,\sigma)$ passing $(t_1,x_1).$ Then the curve $\mathcal C$ given by
$$ x=x(t) =\simuleq { & x_1 +\int_{t_1}^t H_p(\tau,q)d\tau,\ \ t\in [0,t_1] \\
& x_1 +\int_{t_1}^tH_p(\tau,\omega_z(z))d\tau,\ \ t\in [t_1,T],}$$
is exactly the characteristic curve starting from $y=x_1-\int_0^{t_1}H_p(\tau,q)d\tau$ of problem $(H,\sigma)$ on $[0,T]\times\R^n.$

Conversely, if $$\mathcal C':\ x=x(t,y)=y+\int_0^tH(\tau,\sigma_y(y))d\tau$$ is a characteristic curve of type (I) of problem $(H,\sigma)$ on $[0,T]\times\R^n$ starting from $y\in \R^n.$ Then $\sigma_y(y)=q\in \ell(t_1,x_1)$, where $x_1=y+\int_0^{t_1}H(\tau,\sigma_y(y))d\tau,$ and $q=\omega_z(x_1).$ Moreover, $\mathcal C'$ can be rewritten as $$x=x(t,y)=x_1 +\int_{t_1}^tH(\tau,\omega_y(y))d\tau,\ t\in [t_1, T],$$  that is a characteristic curve of the problem $(H,\omega)$ on $[t_1,T]\times\R^n$ starting at $(t_1,x_1).$  The theorem has been proved.
\end{proof}

\subsection{Singularity of viscosity solution}
Since a viscosity solution of the Cauchy problem (3.1)-(3.2) is a semiconvex function, we can apply the results on the propagation of singularity of a  semiconvex (or semiconcave) functions  established in [2, p.85].\vspace*{0.05in}

We will define the semidifferentials of a continuous function. Let $u=u(t,x): \ \Omega \to \R$ and let $(t_0,x_0)\in \Omega.$ For $(h,k)\in \R\times \R^n$, define
$$  \tau (p,q,h,k) =\frac{u(t_0+h, x_0+k)-u(t_0,x_0)-ph -\langle q,k\rangle}{\sqrt{|h|^2 +|k|^2}},$$
$$D^+u(t_0,x_0)=\{ (p,q)\in\R\times\R^{n}\, |\ \limsup_{(h,k)\to (0,0)} \tau (p,q,h,k) \ \le \ 0\}, $$
$$D^-u(t_0,x_0)=\{(p,q)\in \R\times\R^{n}\, | \ \liminf_{(h,k)\to (0,0)} \tau(p,q,h,k)\ \ge \ 0\},$$
where $p\in \R,\ q\in \R^n.$\vspace*{0.05in}

The set $D^+u(t_0,x_0)$ (resp. $D^-u(t_0,x_0)$) is called the {\it superdifferential} (resp. {\it subdifferential}) of $u(t,x)$ at $(t_0,x_0).$\vspace*{0.05in}

In what follows, we denote by $D^*u(t,x)\subset \R\times \R^n$ the set of all {\it reachable gradients} of $u(t,x)$ at $(t,x)$. That means $(p,q)\in D^*u(t,x) $ if and only if there exists a sequence $(t_k,x_k)_k\subset \Omega\setminus \{(t,x)\}$ such that $u(t,x)$ is differentiable at $(t_k,x_k)$ and,
$$(t_k,x_k)\to (t,x), \ (u_t(t_k,x_k), u_x(t_k,x_k))\to (p,q)\  \text{ as}\  k\to \infty.$$

If $u(t,x)$ is semiconvex, then $D^-u(t,x) =\text{co}\, D^*u(t,x),$ the convex hull of $D^*u(t,x)$; see [2, p.54]. \vspace*{0.05in}

Suppose that $u(t,x)$ is the Hopf-type formula (3.3) for the problem (3.1)-(3.2). Then it is a locally Lipschitz function on $\Omega.$\vspace*{0.05in}

Let $(t_0,x_0)\in \Omega.$ Define
$$\mathcal H(t_0,x_0)=\{(-H(t_0,q), q)\ | \ q\in \ell(t_0,x_0)\}.$$
Then a relationship between $D^*u(t_0,x_0)$ and the set $\ell (t_0,x_0)$ is given by the following theorem.

\begin{thm} {\rm([7])} Assume that condition {\rm (A2)} is satisfied on $[0, T]\times \R^n$. In case that $H(t,\cdot)$ is a concave function, we assume in addition that $H(t,\cdot)$ is strictly concave for a.e. $t$ in $(0,T).$ Let $u(t,x)$ be the Hopf-type formula for problem {\rm(3.1)-(3.2)}. Then for all $(t_0,x_0)\in \Omega,$ we have
$$D^*u(t_0,x_0)=\mathcal H(t_0,x_0).$$
\end{thm}
We also denote $\sum (u)$ the set of all singular points of a function $u(t,x)$ (i.e., the points at which $u(t,x)$ is not differentiable). Using an important result on the propagation of singularities of a semiconcave function presented in [2], we have the following:

\begin{thm}
Let $[0,T]=[0,t_1]\cup [t_1,T]$ and suppose that condition {\rm(A2)} is satisfied on each subdomain $[0,t_1]\times \R^n$ and $[t_1,T]\times \R^n.$ Let $(t_0,x_0)\in \sum (u)$ and assume that the $0-$level set of the function $g(p,q)=p+H(t_0,q)$ does not contain any line segment $[(p_1,q_1), (p_2,q_2)]\subset \R^{n+1}$, $(p_1,q_1)\neq (p_2,q_2)$.  Then there exists a Lipschitz arc ${\textbf x}:\ [0,\rho]\to \R^{n+1}$ with ${\textbf x}(0)=(t_0,x_0)$ and ${\textbf x}(s)\in \sum(u)$ for any $s\in  (0,\rho)$ and ${\textbf x}(s)\ne (t_0,x_0)$ for $s >0$ small enough.
\end{thm}

\begin{proof}
Note that the set $\ell(t_0,x_0)$ is compact. Thus, the set
$$D^-u(t_0, x_0)={\text{\rm co}}D^*u(t_0,x_0) ={\text {\rm co}}\mathcal H(t_0, x_0)$$
is also compact. It follows that there exists $(p_0, q_0)\in D^-u(t_0,x_0)$ such that
\begin{equation*}
g(p_0, q_0)=\max_{(p, q)\in D^-u(t_0,x_0)}g(p, q).
\end{equation*}
Since the $0-$level set of function $g(p,q)$ does not contain any line segment and $D^-u(t_0, x_0)$ is not a singleton as $(t_0,x_0)\in \sum (u)$,
$$\displaystyle\max_{(p, q)\in D^-u(t_0,x_0)}g(p,q)=\max_{(p, q)\in D^-u(t_0,x_0)}(p+H(t_0,q))=p_0+H(t_0,q_0)=\alpha >0.$$
 Thus, $(p_0,q_0)\in D^-u(t_0,x_0)\setminus D^*u(t_0,x_0).$ For any $n\in \N,$ we see that $$1/n +p_0 +H(t_0,q_0) >\alpha.$$ Therefore, $(p_0+1/n, q_0)\notin D^- u(t_0,x_0).$ This means that $$(p_0,q_0)\in \partial D^-u(t_0,x_0)\setminus D^*u(t_0,x_0).$$ Applying \cite[Theorem 4.2.2]{[2]}, we obtain the desired result.
\end{proof}
\begin{rem} 1. From Theorem 4.5, we do not know whether the singular point propagates forward or backward in time.\vspace*{0.05in}

2. If we use method of characteristics or Hopf-type formula to define generalized solutions of the Cauchy problem of Hamilton-Jacobi equation, we see that the singular points of solutions may propagate forward in time $t.$  Nevertheless, if the structure of the Hamiltonian (convexity, for example) changes in time, then the singularities of solution may cease.\vspace*{0.05in}

{\bf Example.} Consider the following Cauchy problem:

   $$ \simuleq{u_t +(2t-1)u_x^2 \ =\  & 0, \ (t,x)\in (0, 2)\times \R, \\
    u(0,x)\ =\ & \ |x|,\ \ x\in \R.}$$

Since the Hamiltonian $H(t,p) =(2t-1)p^2$, as a function of $p,$ is concave if $t\in [0,1/2] $ and convex if $t\in [1/2, 2],$ the Hopf-type formula of the problem
$$u(t,x)=\max_{p\in [-1,1]}\{xp-(t^2-t)p^2\}$$ is not a viscosity solution of the problem; see [7].\vspace*{0.05in}

On the other hand, applying Theorem 3.3, we obtain a viscosity solution of the problem as follows:
$$u(t,x)=\simuleq{&\max_{p\in [-1,1]}\{xp-(t^2-t)p^2\}, \ \ (t,x)\in [0,\frac12]\times\R,\\
 & \max_{p\in [-1,1]}\{xp-(t^2-t)p^2\}+\frac 14,\  \ (t,x)\in [\frac12, 2]\times\R.}$$
A simple calculation gives us
$$u(t,x) =\simuleq{ &|x| -(t-\frac 12)^2 +\frac 14,  \ \ (t,x)\in [0,\frac 12]\times\R,\\
&\frac{x^2}{4(t-\frac12)^2} +\frac 14, \ \ |x|\le 2(t-\frac 12)^2,\ \frac 12 <t\le 2,\\
& |x|-(t-\frac 12)^2 +\frac14, \ \ |x| \ge 2(t-\frac12)^2,\ \frac 12 < t\le 2.}$$
At the points $(\frac12,x)$, we have
$u(\frac12,x)=|x|+\frac14,\ \ x\in \R.$\vspace*{0.05in}

It is obvious that $(\frac12,0)$ is a singular point of the solution and $\ell (\frac12,0)=\{-1,1\}.$ Therefore, $D^*u(\frac12,0)=\{(0,-1),(0,1)\}$ and $D^-u(\frac12,0)=\{0\}\times [-1,1].$ Then $$\partial D^-u(\frac12,0)\setminus D^*u(\frac12,0)=\{0\}\times(-1,1) \ne \emptyset.$$ By Theorem 4.5, the singularity of $u(t,x),$  say, at $(\frac 12, 0)$ propagates. Nevertheless, for $t\in (\frac 12, 2]$, the Hamiltonian is strictly convex. Thus, the solution of the problem is of $C^1((\frac 12,2]\times\R).$ This proves that the singularity of the solution propagates backward.
\end{rem}


\end{document}